\documentclass{article}
\usepackage{amsfonts, amsmath} % used for R in Real numbers
\usepackage{amssymb}
\usepackage{amsthm}
\usepackage{mathrsfs}
\usepackage[english]{babel}
\usepackage[utf8]{inputenc}
\usepackage[T1]{fontenc} 
\usepackage{enumitem}
\usepackage{hyperref}
\usepackage[left=2.5cm, right=2.5cm]{geometry} 
\usepackage{color}
\usepackage{authblk}

\begin{document}

\newtheorem{definition}{Definition}[section]
\newtheorem{theorem}{Theorem}[section]
\newtheorem{example}{Example}[section] 
\newtheorem{remark}{Remark}[section] 
\newtheorem{corollary}{Corollary}[section]
\newtheorem{lemma}{Lemma}[section]
\newtheorem*{T. lemma}{Technical lemma }  
\newtheorem{proposition}{Proposition}[section]
\numberwithin{equation}{section}

\renewcommand{\thefootnote}{\fnsymbol{footnote}} 
\renewcommand{\thefootnote}{\arabic{footnote}}

\title{A priori bounds and multiplicity results for slightly superlinear and sublinear elliptic  $p$-Laplacian equations}

\author[,1]{Zakariya Chaouai \thanks{Corresponding author:  \texttt{zakariya.chaouai@cea.fr}}} 
\author[,1]{Mohamed Tamaazousti \thanks{\texttt{mohamed.tamaazousti@cea.fr}}}
 
\affil[1]{Université Paris-Saclay, CEA, List, F-91120, Palaiseau, France}

\date{}
%\email{}
%	\address{}	
%    \subjclass is required.
%		\subjclass[2010]{Primary 35J66, 35A15, 	35B38 }
\maketitle

\begin{abstract}
We consider the following problem $ -\Delta_{p}u= h(x,u) \mbox{ in }\Omega$, $u\in W^{1,p}_{0}(\Omega)$,
        where $\Omega$ is a bounded domain in $\mathbb{R}^{N}$, $1<p<N$, with a smooth boundary. In this paper we assume that $h(x,u)=a(x)f(u)+b(x)g(u)$ such that  $f$ is regularly varying of index $p-1$ and superlinear at infinity. The function $g$ is a $p$-sublinear function at zero. The coefficients $a$ and $b$ belong to $L^{k}(\Omega)$ for some $k>\frac{N}{p}$ and they are  without sign condition. Firstly, we show a priori bound on solutions, then by using variational arguments, we prove the existence of at least two nonnegative solutions. One of the main difficulties is that the nonlinearity term $h(x,u)$ does not satisfy the standard Ambrosetti and Rabinowitz condition. 
\end{abstract}

%\subjclass[2010]{Primary 35J66, 35A15,  35B38}

\section{Introduction} 
Let $\Omega$ be a bounded domain in $\mathbb{R}^{N}$ ($N\geq 2$) with a smooth boundary $\partial\Omega$. In this paper, we are concerned by the following boundary value problem:
$$(P) \begin{cases}                                                             
    -\Delta_{p}u= h(x,u) & \ \ \mbox{ in }\Omega,\\                                                  
    u=0 & \ \ \mbox{ on }   \partial\Omega,                                                                 
    \end{cases}$$
where $\Delta_{p}u= \mbox{div}(\vert\nabla u \vert^{p-2}\nabla u)$ is the $p$-Laplacian operator,  $1<p<N$ and $h : \Omega\times [0,\infty)\to [0,\infty)$ is a Carathéodory function that satisfies some suitable conditions. 

%Starting from the existing works on this problem, we will motivate our choices for these conditions which lead us to the model that we introduce in the next section. 

%The coefficients $a$ and $b$ belong to $L^{k}(\Omega)$ with $k>\frac{N}{p}$ and they are without sign condition. It will be assumed throughout the paper that these functions satisfy some specific conditions. 
 
We start by stating the existing works on this elliptic problem. Then we will motivate the condition choices we made on the parameters of this problem, which culminate in the model we introduce in the next section. Indeed, in the literature, several works studied different instance of  problem $(P)$. For the case  $p=2$, in \cite{De Figueiredo3} De Figueiredo et al. proved  the existence of nontrivial solutions to problem $(P)$ in $\mathbb{R}^{2}$, where $h(x,u)$ has exponential growth. In \cite{Bartsch}, Bartsch and Wang proved the multiplicity of nontrivial solutions in the case where $h(x,u)=h(u)\in C^{1}(\Omega)$ grows superlinearly but subcritically at infinity and $h^{'}(0)<\lambda_{2}$, and where $\lambda_{2}$ is the second small eigenvalue of $-\Delta$ on $\Omega$. Rec\^{o}va and Rumbos  proved in \cite{Recova} the multiplicity of solutions in the case where $h(x,u)$ has a polynomial growth and satisfies the nonquadraticity condition introduced by Costa and Magalhães in \cite{Costa}.   In \cite{De Figueiredo}, De Figueiredo et al. showed, by using the variational and sub and super solutions methods,  the existence and the multiplicity of positive solutions in the case where $h(x,u)$ is locally superlinear and sublinear. Moreover, these results were extended to the $p$-Laplacian case by the same authors in \cite{De Figueiredo2}. Recently, De Figueiredo et al. showed in \cite{De Figueiredo4} the existence of at least two solutions in the case where $h(x,u)$ is locally $p$-sublinear at zero.   

It is worth mentioning that, in all the papers mentioned above, the authors assumed, among other conditions, that the nonlinearity term $h(x,s)$ satisfies the well known Ambrosetti-Rabinowitz (AR) condition that was introduced for the first time in \cite{AR}. The condition requires the following in the case of the $p$-Laplacian operator;
  $$\mbox{ there exist }  \; \theta >p \; \mbox{ and } \; s_{0}>0 \;  \mbox{ such that }\;   0<\theta H(x,s)\leq s h(x,s),\;   \mbox{ as } \; \forall s>s_{0}.$$
It implies the existence of two positive constants $C_{1}$ and $C_{2}$ such that 
$$ H(x,s)\geq C_{1} s^{\theta} -C_{2}, \; \; \forall s\geq 0.$$
In the literature, this condition is the main tool to prove the existence of solutions to elliptic problems with variational structure. It serves in particular to prove the boundedness of Palais-Smale sequence of the energy functional associated with such problems.  Nevertheless, this condition is somewhat restrictive and not
being satisfied by many nonlinearities. 

We can mention for instance the nonlinearity  $h(x,s)= s^{p-1}\ln(1+s)^{q}$, where $q>0$, that was taken from \cite{Chaouai,Coster22}, which instead verify that, for any  $\theta>p$, $H(x,s)/s^{\theta}\to 0$ as $s\to +\infty$. In our present setting as well, the nonlinearity  $h(x,s)$ that we are considering, does not satisfy the (AR) condition.    

Many recent research have been made to drop the (AR) condition, we refer for instance to \cite{De Figueiredo, Hsu, Iturriaga1}. In these works,  the authors studied different boundary value problems. In \cite{De Figueiredo} De Figueiredo et al. studied the problem $(P)$ in the case where $p=2$ and the nonlinearity term  $h(x,u)=  a(x) u^{q}+ b(x)u^{r}$, where the coefficients $a(x)$, $b(x)$ belong to $L^{\infty}(\Omega)$ with $a(x)$ is nonnegative and $b(x)$ is allowed to change sign.  They proved that if $a(x)$ and $\Vert b\Vert_{L^{\infty}(\Omega)} - b(x)$ are suitably small, then the problem $(P)$ has at least two positive solutions. In a related task, in \cite{Hsu} Hsu deals with the same nonlinearity term but in the $p$-Laplacian case. Moreover, he assumes that  $r=\frac{Np}{N-p}$, the coefficients $a(x)$ and $b(x)$ are continuous on $\Bar{\Omega}$ and are positive somewhere but may change sign. He proved that, if $0<q<p-1<N-1$ and  $a(x)$ is suitably small, then the problem $(P)$ has at least two positive solutions. On the other hand, in \cite{Iturriaga1} Iturriaga et al. considered more general setting. Notably, they supposed that $h(x,u)$ grows as $u^{p-1}$ near zero and has a $p$-superlinear growth at inﬁnity. They proved under other conditions that $(P)$ has at least two positive solutions. We point out, that the class of nonlinearities considered in these works belong to  pure-power or power-like cases which is a quite restrictive class of functions. Moreover, the assumptions considered on the coefficients $a(x)$ and $b(x)$ are very strong. In this paper, the nonlinearity $h(x,s)$ includes a larger class of functions. Namely, part of  our nonlinearity  belongs to the class of functions that is regularly varying of index $p-1$ (at infinity), (see section 2). This type of functions were introduced for the first time in \cite{Karamata}, (see also \cite{Seneta}).   

%We point out, in  this work our nonlinearity term $h(x,u)$ contain some coefficients without sign condition (see section 2), because of this assumption and the failure of the (AR) condition the study of the multiplicity of solutions of the problem $(P)$ become more difficult.  Before moving forward, we give a review of some results related to this context.  In the case where $h(x,u)= a(x) u^{q}+ b(x)u^{r}$,  

Recently, in the case $p=2$,  \cite{Chile,Costa2} considered the case  where the nonlinearity term  is  regularly varying. More specifically,  in \cite{Chile} Garc\'{\i}a-Meli\`an et al. studied the problem $(P)$ in the case where $h(x,u)=a(x)f(u)$, such that $f$ is  regularly varying of index $1$ and slightly superlinear and $a(x)$ may change of sign. They proved the existence of at least one positive solution. Moreover, in \cite{Costa2}  Costa et al.  proved the existence of at least one positive solution. They studied as well the bifurcation in the case where $h(x,u)=\lambda a(x)(f(u)-l)$ where $f$ is subcritical, superlinear at infinity, and regularly varying of index $q>1$,  $\lambda$ and $l$ are positive parameters. In these two papers, it was supposed  among other conditions that  the coefficient $a(x)$ is continuous on $\overline{\Omega}$, which is a strong assumption. Furthermore, in  \cite{Jeanjean2} Jeanjean and Quoirin consider the case where $h(x,u)= a(x)f(u)+ b(x)g(u)$ such that $f(u)=(1+u)\ln(1+u)$ and $g(u)=1+u$. The coefficient $a(x)$ is without sign condition and the coefficient  $b(x)$ is nonnegative  function and both belong to $L^{k}(\Omega)$ with $k>\frac{N}{2}$. They proved that if $a$ is suitably small, then the problem $(P)$ has at least two solutions. In \cite{Coster2} De Coster and Fern\'andez  proved the same results where $b$ may change of sign. Moreover, in \cite{Chaouai} Chaouai and Maatouk generalized the result given in \cite{Jeanjean2} to the $p$-Laplacian operator, where $a(x)$ is nonnegative and $b(x)$ changes its sign.  We highlight that $f(u)$ is an example of functions that is regularly varying of index 1. In addition, the arguments used by the authors in \cite{Chaouai,Coster2,Jeanjean2} to prove the multiplicity of solutions are based on the explicit form of $f$ and $g$.

Our  aim in this paper is to show the existence of at least two different nonnegative weak solutions to the problem $(P)$. It is worth mentioning that the form of our model (see section 2) is inspired by \cite{Chaouai,Cirstea0,Cirstea,Coster2,  Jeanjean2}. An important feature of our study lies in the incorporation of  potentials and weighted nonlinearities, thereby embracing classes of the stationary nonlinear Schrodinger equations. Moreover, our model corresponds in particular cases to the time-independent case of the nonlinear Hamilton-Jacobi equations studied in \cite{Galaktionov} (see Section 5). On the other hand, the improvements given in this work are as follows;  The first improvement, the nonlinearity term that we are considering belongs to a larger class of functions. Notably, it composes of two general functions; $i)$ One of them is regularly varying of index $p-1$, which is inspired by \cite{Chile}. $ii)$ Second term that we consider a general $p$-sublinear function at zero as in \cite{De Figueiredo2, De Figueiredo4}. The second improvement, lies on the fact that we assume weak conditions on the coefficients of the nonlinearity $h(x,s)$. More specifically, instead of assuming that the coefficients are continuous like in the recent works \cite{Chile,Costa2}, we are only assuming that our coefficients belong to $L^k(\Omega)$,  $k>\frac{N}{p}$, and with no sign conditions. The last but not least improvement to mention is that our problem involves the $p$-Laplacian operator, with $1<p<N$, unlike \cite{Chile,Coster2, Jeanjean2}. Furthermore, to the best of our knowledge, our results are new even in the Laplacian case i.e., $p=2$.
 
The rest of the paper is organized as follows. In Section 2, we explicitly give the model that we will deal with, then we state the assumptions and our main results. In Section 3, we recall and state some preliminary results that will be useful for the rest of the paper. Section 4 is devoted to prove of our main results. Finally, Section 5 is concerned to some applications of the main results given in Section 2.

\subsection*{Notations} 
        \begin{enumerate}
                \item [1)] The Lebesgue norm $(\int_{\Omega} |u|^{p})^{\frac{1}{p}}$ in $L^{p}(\Omega)$ is denoted by $\|.\|_{p}$ for $p\in [1,+\infty[$. The norm in $L^{\infty}(\Omega)$ is denoted by $\
\|u\|_{L^{\infty}(\Omega)}:=ess \sup_{x\in\Omega}|u(x)|$. The H\"{o}lder conjugate of $p$ is denoted by $p^{\prime}$.
                \item [2)]The spaces $W_{0}^{1,p}(\Omega)$ and $W^{-1,p'}(\Omega)$  are equipped with  Poincaré norm and the dual norm, $\|u\|:=(\int_{\Omega} |\nabla u|^{p})^{\frac{1}{p}}$ and $\|\cdot\|_{\ast}:=\|\cdot\|_{W^{-1,p'}(\Omega)}$ respectively.
                \item [3)]We denote by $B_{R}(x_{0})$ the ball of radius $R$ centered at $x_{0}$ and $\partial B_{R}(x_{0})$ its boundary. 
                \item[4)] for $u\in L^{1}(\Omega)$, we define $u^{+}= \max(u,0)$ and $u^{-}= \min(-u,0)$. 
                \item[5)] We denote by $C,C_{i}$ any positive constants which are not essential in the arguments and which may vary from one line to another.
        \end{enumerate}

%Among the aim goals of this work is to improve the results given in \cite{Chile} to the case where the coefficient belong $L^{k}(\Omega)$ with $k>\frac{N}{p}$. The main difficulty here is to establish a priori estimates without assuming the coefficients are continuous on $\Bar{\Omega}$. Indeed, this assumption seems necessary in \cite{Chile}, in order to apply the the Gidas-Spruck rescaling argument which is introduced in \cite{Gidas}.  We recall also this method have been extensively used when the nonlinearities $h(x,u)$ is asymptotic to a power near $\infty$, see for example \cite{Amann, Iturriaga,Phan, Ruiz} and references therein. In our case we will exploit the new method used in the recent work \cite{Coster1} where the authors show a priori estimates for the special case $p=2$.  Especially, this method is based in using the boundary weak Harnack inequality. The majority part of their work is devoted to show this great inequality.  However, to use this method we shall  extend some results to the $p$-Laplacian setting.  Therefore,   we point out that another goal of this work is to prove the existence of at least two nonnegative solutions for a more generalized problem with a weak assumptions where all type of problems given in recent papers \cite{Coster2,Jeanjean2} will be include in our model. Finally, to the best of our knowledge our results are new even in the case $p=2$.   

\section{Models, assumptions and main results}
In this section, we will give the explicit form of our nonlinearity term $h(x,u)$ of the problem $(P)$, afterwards we state our assumptions and our main results. Next, we will  give some examples of the considered nonlinearities functions that satisfy our hypothesis in each result.
 
 In this paper, our goal is to show the existence of at least two nonnegative nontrivial solutions of the following problem 
$$(Q) \begin{cases}                                                             
    -\Delta_{p}u= a(x)f(u)+ b(x)g(u) & \ \ \mbox{ in }\Omega,\\                                                  
    u\in W_{0}^{1,p}(\Omega),                                                                  
    \end{cases}$$
where the functions $a$ and $b$ satisfy the following conditions
\begin{equation}\label{hypothse} 
\tag{$\mathcal{A}_{a,b}$}
\begin{cases}                                                             
    a = a^{+}-a^{-}, b = b^{+}-b^{-} \text{such that } a^{+}, a^{-}, b^{+} \in L^{k}(\Omega) \text{ for some } k>\frac{N}{p} \text{ and } b^{-}\in L^{\infty}(\Omega),\\                                                  
    a^{+}(x)a^{-}(x)=0 \text{ a.e. in }  \Omega,\\ 
    \Omega^{+}:= \text{supp}(a^{+}) \text{ such that }   \text{mes}(\Omega^{+})>0, \\
    \text{there exists an } \epsilon>0 \text{ such that } a^{-}=0 \text{ in } \{x\in\Omega : d(x,\Omega^{+})<\epsilon\}. 
    \end{cases} 
\end{equation}
We point out, the same assumptions have been considered in the recent paper \cite{Coster1}. We note that the condition $a^{-}=0$ in  $\{x\in\Omega : d(x,\Omega^{+})<\epsilon\}$ for some $\epsilon>0$, it's called "thick zero set" which was introduced for the first time in \cite{Alama}.    
    
Through this paper we assume that $f, g: [0,\infty)\to[0,\infty)$ to be continuous functions and satisfy the following assumptions; 
%\begin{equation}
%  \tag{Super Fun Equation}
%  y = 3x
%  \label{eqn:super}
%\end{equation}
\begin{itemize}
    \item $f$ is regularly varying of index $p-1$ at infinity, i.e. 
    \begin{equation}\label{RVq}
    \tag{$\mathcal{A}^{1}_{f}$}
    \lim_{s\to +\infty}  \frac{f(\lambda s)}{f(s)}=\lambda^{p-1}, \text{ for every } \lambda>0.    
    \end{equation}
\item  $f$ is $p$-superlinear at infinity, i.e.
\begin{equation}\label{limf2}
\tag{$\mathcal{A}^{2}_{f}$}
\lim_{s\to +\infty}\frac{f(s)}{s^{p-1}}= +\infty.    
\end{equation}
\item $g$ is $p$-asymptotically linear at infinity, i.e. 
\begin{equation}\label{limg2}
\tag{$\mathcal{A}^{1}_{g}$}
\lim_{s\to +\infty} \frac{g(s)}{s^{p-1}}= l_{1}, \; \;  \text{where} \; \;   l_{1}\in [0,\infty).    
\end{equation}
\end{itemize}

For more information and developments on regularly varying of index $p-1$  functions we refer the reader to  \cite{Chile, Karamata,Seneta}.

In order to prove the multiplicity of nonnegative solutions for $(Q)$, we will use variational arguments. As in \cite{Chile} the main ingredients of our arguments are some a priori bounds on any nonnegative weak solution of a slightly more general version  of $(Q)$, which is as follow,  
$$(Q_{\lambda}) \begin{cases}                                                             
    -\Delta_{p}u= (\lambda a^{+}(x)-a^{-}(x))f(u) + b(x)g(u) & \ \ \mbox{ in }\Omega,\\                                                   
   u\in W_{0}^{1,p}(\Omega),                                                                  
    \end{cases}$$
where $\lambda\in [1/2,1]$. As mentioned in the introduction, in \cite{Chile} it was assumed that the coefficient $a(x)$ is continuous on $\Bar{\Omega}$ and $b\equiv 0$. The method that was applied is the well known  Gidas-Spruck rescaling argument which was introduced in \cite{Gidas}.  We recall that this method have been extensively used when the nonlinearities term is asymptotic to a power near $\infty$, see for instance \cite{Amann, Iturriaga, Ruiz} and references therein. In our case, we will exploit the new method used in the recent work \cite{Coster1} where the authors show a priori estimates for the special case $p=2$.  Especially, this method is based on using the boundary weak Harnack inequality.  However, to use this method we shall  extend some results to the $p$-Laplacian setting, see Section 3.

%that to establish a priori estimates without assuming the coefficients are continuous on $\Bar{\Omega}$. Indeed, this assumption seems necessary in \cite{Chile}, in order to apply the the Gidas-Spruck rescaling argument which is introduced in \cite{Gidas}.  We recall also this method have been extensively used when the nonlinearities $h(x,u)$ is asymptotic to a power near $\infty$, see for example \cite{Amann, Iturriaga,Phan, Ruiz} and references therein. In our case we will exploit the new method used in the recent work \cite{Coster1} where the authors show a priori estimates for the special case $p=2$.  Especially, this method is based in using the boundary weak Harnack inequality. The majority part of their work is devoted to show this great inequality.  However, to use this method we shall  extend some results to the $p$-Laplacian setting.

Now, let us state our first result which is concerning  a priori estimates of all nonnegative solutions of $(Q_{\lambda})$.    
\begin{theorem}[A priori bounds]\label{a priori estimate} 
 Under the assumptions \eqref{hypothse},\eqref{RVq}, \eqref{limf2},  and \eqref{limg2}, there exists a positive constant $C$ such that for any $\lambda\in [1/2, 1]$ and every nonnegative weak solution $u$ of $(Q_{\lambda})$ satisfies   $$\Vert u\Vert_{L^{\infty}(\Omega)}\leq C.$$  
\end{theorem}

Concerning the existence of the multiplicity result,  other hypothesis have to be added to the one of Theorem \ref{a priori estimate} for $f$ and $g$ which are as follow;
\begin{itemize}
\item $f$ is $p$-superlinear at zero, i.e. 
\begin{equation}\label{limf1}
\tag{$\mathcal{A}^{3}_{f}$}
\lim_{s\to 0^{+}} \frac{f(s)}{s^{p-1}}= 0,
\end{equation}
\item $g$ is $p$-sublinear at zero, i.e.  
\begin{equation}\label{limg1}
\tag{$\mathcal{A}^{2}_{g}$}
\lim_{s\to 0^{+}} \frac{g(s)}{s^{p-1}}= +\infty,
\end{equation}  
\end{itemize}
The assumption \eqref{limf2} will be used to prove the geometrical structure of the functional $I_{\lambda}$, defined in Section 4, associated to the problem $(Q_{\lambda})$ and this condition implies that $f(0)=0$. On the other hand, the assumption \eqref{limg1} will be used to prove the existence of the second weak solution. In addition, this condition implies that $g(0)\geq 0$. In our results we will distinguish between the case where $g(0)=0$ and the case $g(0)>0$.  

Due to Theorem \ref{a priori estimate} and to above assumptions, we can have the following multiplicity results of the problem $(Q)$.  
\begin{theorem}[First multiplicity result]\label{maintheo1}
Under the assumptions  \eqref{hypothse}, \eqref{RVq},\eqref{limf2}, \eqref{limf1}, \eqref{limg2} and \eqref{limg1}, we also assume that $\Vert b^{+} \Vert_{L^{k}(\Omega)}$ is suitably small. 
\begin{enumerate}
\item If $g(0)=0$, then the problem $(Q)$  has at least two nonnegative  nontrivial weak solutions in $W_{0}^{1,p}(\Omega)\cap L^{\infty}(\Omega)$. 
\item If $g(0)>0$ and $b\gneqq 0$, then the problem $(Q)$  has at least two positive weak solutions in $W_{0}^{1,p}(\Omega)\cap L^{\infty}(\Omega)$.
\end{enumerate}

\end{theorem}

Various functions occurring in various works are included as models for the boundary value problem $(Q)$ as one can see from next examples.
\begin{example}
Some functions that satisfy \eqref{RVq}-\eqref{limf1} are as follow,
\begin{itemize}
\item[(i)] $f(s)= s^{p-1}(\log(1+s))^{l}$, \cite{De Figueiredo4}
\item[(ii)] $f(s)= s^{p-1}(\exp(\log(1+s)^{l})-1)$, \cite{Chile}
\end{itemize}
 where $l>0$.
\end{example}
\begin{example}\label{example g}
Some functions that satisfy \eqref{limg2} and \eqref{limg1} are as follow
\begin{itemize}
\item[(i)] $g(s)= 1+s^{p-1}$, \cite{Costa1}
\item[(ii)] $g(s)= \exp(\frac{(p-1)s}{\epsilon +s})$, \cite{Abdellaoui}
\end{itemize}
 where $\epsilon>0$. 
\end{example}
We stress that the example \ref{example g} $(ii)$ appears in combustion theory or, more generally, describes reactions of Arrhenius type in the case $p=2$, (for more details on this model, see \cite{Bebernes}).
%g(s)= (1+s)^{q}e^{-s}, 

For the next result, we will add another assumption on the coefficient $a^{+}$. However, we will make weaker assumptions on $f$ and $g$. Specifically, instead of assuming the assumptions \eqref{limf1} and \eqref{limg1}. We assume the following assumptions.
\begin{itemize}
\item There exists $r\in (0,p-1)$  such that 
\begin{equation}\label{limf3}
\tag{$\mathcal{A}^{3'}_{f}$}
\lim_{s\to 0^{+}} \frac{f(s)}{s^{r}}= l_{2} \; \;  \text{where} \; \;   l_{2}\in (0,\infty).
\end{equation}
\item $g$ is either $p$-sublinear or $p$-asymptotically linear at zero, i.e.  
\begin{equation}\label{limg3}
\tag{$\mathcal{A}^{2'}_{g}$}
\lim_{s\to 0^{+}} \frac{g(s)}{s^{p-1}}= l_{3} \; \;  \text{where} \; \;   l_{3}\in (0,\infty].
\end{equation}
\end{itemize}
 \begin{theorem}[Second multiplicity result]\label{maintheo2} 
Under the assumptions \eqref{hypothse}, \eqref{RVq}, \eqref{limf2},\eqref{limf3}, \eqref{limg2} and \eqref{limg1}, we also assume that $\Vert a^{+} \Vert_{L^{k}(\Omega)}$ and $\Vert b^{+} \Vert_{L^{k}(\Omega)}$ are suitably small. Then the conclusions of Theorem \ref{maintheo1} hold.
\end{theorem}
\begin{example}\label{last f}
Some functions that satisfy \eqref{RVq}, \eqref{limf2} and \eqref{limf3} are as follow,
\begin{itemize}
\item[(i)] $f(s)= (1+s)^{p-1}(\log(1+s))^{l}$,  \cite{Chaouai}
\item[(ii)] $f(s)=(1+s)^{p-1}(\log(1+\log(1+s))^{l}$, \cite{Chile}
%\item[(ii)] $f(s)= s^{p-1}(\log(K+s))^{l}$,\cite{Souplet}
\end{itemize}
 where $l>0$.
\end{example}
We point out that the above examples, satisfy the property \eqref{limf3}, for both cases, when $r=l$.  
\begin{example}\label{last f}
Some functions that satisfy \eqref{limg2},  \eqref{limg3} are as follow,
\begin{itemize}
\item[(i)] $g(s)= s^{q}$, \cite{Dinca}  
\item[(ii)] $g(s)=\frac{s^{q}}{(s+\epsilon)^{q+\alpha}} $, \cite{Diaz} 
\end{itemize}
where $0\leq q\leq p-1$ and $\epsilon, \alpha>0$.
\end{example}
\begin{remark}
 If $u$ is a weak solution of $(Q_{\lambda})$ then from Lemma \ref{regul} below, $u\in C^{0,\alpha}(\Bar{\Omega})$ for some $\alpha\in (0,1)$. Beside, if the coefficients $a$ and $b$ belong to $L^{\infty}(\Omega)$ we can use   {\cite[Theorem 1]{Lieberman}} and we deduce directly from Theorem \ref{maintheo1}, or Theorem \ref{maintheo2}, that $u\in C^{1,\alpha}(\Bar{\Omega})$ for some $\alpha\in (0,1)$. 
 \end{remark} 
 \begin{remark}
In the assumption \eqref{limg3} if $l_{3}<+\infty$, then for $s$ small enough we obtain that $g(s)<C s^{p-1}$ where $C$ is a positive constant. By the same arguing used in Lemma \ref{positivity} below, we deduce directly that any nonnegative nontrivial weak solution of $(Q)$ is positive.
\end{remark}
\begin{remark}
The Theorem \ref{maintheo2} is still satisfied even if $r=p-1$ in the assumption \eqref{limf3} but only in the case where $g$ is $p$-sublinear at zero, namely $l_{3}=+\infty$, (see the proof of Theorem \ref{maintheo2}). In our application, see Subsection $5.2$, this case will be included. 
\end{remark} 

\section{Preliminary results}  
In this section, we present some definitions and  results which will play a important role throughout this work. We note that from now we extend $f$ and $g$ for $s<0$ by putting $f(s)=0$ and $g(s)=g(0)\geq 0$. 

Let us begin with the following result which concerns the functions that are regularly varying of index $p-1$. This result will be useful for our proofs.   
\begin{proposition}\label{prop RV}
We assume that $f$ is  continuous on $[0,\infty)$ satisfies \eqref{RVq}, then for every $\delta>0$, there exists a positive  constant $C$ such that
    $$ 
    f(s)\leq C(1+s^{p-1+\delta}), \; \;     s\in [0, \infty)
    $$
    where $C$ depends on $\delta$.
\end{proposition}
\begin{proof}
This property was done in \cite{Chile} in the case $p=2$, but it can be easily extended to $p>1$.
\end{proof}
In the following result we will give some properties of  weak solutions of the problem $(Q_{\lambda})$. 
\begin{lemma}\label{regul}
Under the assumptions \eqref{hypothse}, \eqref{RVq}, \eqref{limg2}, we have
\begin{itemize}
\item[(i)]  Every weak solution of $(Q_{\lambda})$ belongs to $L^{\infty}(\Omega)$.
\item[(ii)] Every weak solution of $(Q_{\lambda})$ belongs to $C^{0,\alpha}(\Bar{\Omega})$ for some $\alpha\in (0,1)$.
\end{itemize}
\end{lemma}
\begin{proof}
\begin{itemize} 
\item[(i)] We have from Proposition \ref{prop RV} that, for every $\delta>0$ there exist a positive constant $C_{1}$ such that 
$$f(s)\leq C_{1}(1+s^{p-1+\delta}), \; \;     s\in [0, \infty)
    $$
Since $g$ is continuous on $[0,\infty)$, thus from \eqref{limg2} we deduce that,  there exist positive constant $C_{2}$ such that 
$$g(s)\leq C_{2}(1+s^{p-1}).$$
Hence, by using {\cite[Theorem IV-7.1]{Ladyzhenskaya}} (or {\cite[Theorem 2.4]{Pucci}} ) we deduce the result. 
\item[(ii)]By using (i) its follows directly from {\cite[Theorem IV-2.2]{Ladyzhenskaya}}.
\end{itemize}
\end{proof}
Now, let us state the definition of super-solution and sub-solution of the following general  boundary value problem
\begin{equation}\label{probgen}
 \begin{cases}                                                             
    -\Delta_{p}u+ F(x,u)=h(x)  & \ \ \mbox{ in }\Omega,\\                                                  
   u\in W_{0}^{1,p}(\Omega),                                                             
    \end{cases} 
\end{equation}    
such that $h\in L^{1}(\Omega)$ and $F: \Omega\times\mathbb{R}\to \mathbb{R}$ is a Caratheodory function. 
\begin{definition}
We say that $u\in W^{1,p}(\Omega)$ is a weak super-(sub-)solution of \eqref{probgen} if for all $\varphi\in W_{0}^{1,p}(\Omega)$ with $\varphi\geq 0$, we have :
$$\int_{\Omega} \vert\nabla u \vert^{p-2}\nabla u\nabla \varphi + \int_{\Omega} F(x,u)\varphi dx \geq (\leq) \int_{\Omega} h(x)\varphi,$$
$$u\geq ( \leq ) \; 0, \; \;  \mbox{on } \partial\Omega.$$
A function $u$ is a weak solution of \eqref{probgen} if it is a super-solution and a sub-solution.
\end{definition}
Next, we state the following comparison principle which will have an important role for the proof of our auxiliaries lemmas and results. Let us consider the following boundary value problem  
\begin{equation}\label{probpartbis}
 \begin{cases}                                                             
    -\Delta_{p}u+ c(x) \vert u\vert^{p-2}u=0  & \ \ \mbox{ in }\Omega,\\                                                  
    u=0 & \ \ \mbox{ on }   \partial\Omega,                                                                 
    \end{cases}   
\end{equation}   
\begin{lemma}\label{weak comparaison lemma}
Let $u,v\in W^{1,p}(\Omega)$ be a super-solution and a sub-solution of \eqref{probpartbis} respectively. We assume $c\in L^{\infty}(\Omega)$ is a nonnegative function. Then we have $v\leq u$
\end{lemma}
\begin{proof}
 See {\cite[Lemma 3.1]{Tolksdorf}}
\end{proof} 

Now, we consider the following boundary value problem.
\begin{equation}\label{probpart}
 \begin{cases}                                                             
    -\Delta_{p}u+ c(x) \vert u\vert^{p-2}u=h(x)  & \ \ \mbox{ in }\Omega,\\                                                  
    u=0 & \ \ \mbox{ on }   \partial\Omega,                                                                 
    \end{cases}
\end{equation} 
with $u \in W_{0}^{1,p}(\Omega)$. Here we suppose  the coefficients $c(x)$ and $h(x)$ satisfy the following assumptions 
\begin{equation}\label{partcond}
    c, \;  h \in L^{k}(\Omega) \text{ for some } k>\frac{N}{p}.
\end{equation} 
We point out, from {\cite[Theorem 13]{Dinca}}, that the problem \eqref{probpart} has at least one solution. In the next lemmas we state the local maximum principale  and the boundary local maximum principale, which we will be needed in the next section. 
\begin{lemma}\label{lmp} 
Under the assumption \eqref{partcond}. Let $u\in W^{1,p}(\Omega)$ be a sub-solution of \eqref{probpart}. For any ball $B_{2R}(x_{0})\subset\Omega$ and any $q\in(0, p]$, there exists $C= C(k,q,p,N,R, \Vert c \Vert_{L^{k}(B_{2R}(x_{0}))})>0$ such that
$$\sup_{B_{R}(x_{0})} u^{+} \leq C \Biggr[\left(\int_{B_{2R}(x_{0})}(u^{+})^{q}\right)^{\frac{1}{q}} + \Vert h^{+} \Vert^{\frac{1}{p-1}}_{L^{k}(B_{2R}(x_{0}))}\Biggr].$$
\end{lemma}
\begin{proof}
 See {\cite[Corollary 3.10]{Maly}}.
\end{proof}
\begin{lemma}\label{blmp} 
Under the assumption \eqref{partcond}. Let $u\in W^{1,p}(\Omega)$ be a sub-solution of \eqref{probpart} and let $x_{0}\in \partial\Omega$. For any $R>0$  any $q\in(0, p]$, there exists  $C= C(k,q,p,N,R,\Omega ,\Vert c \Vert_{L^{k}(B_{2R}(x_{0})\cap\Omega)})>0$ such that  
$$\sup_{B_{R}(x_{0})\cap\Omega} u^{+} \leq C \Biggr[\left(\int_{B_{2R}(x_{0})\cap\Omega}(u^{+})^{q}\right)^{\frac{1}{q}} + \Vert h^{+} \Vert^{\frac{1}{p-1}}_{L^{k}(B_{2R}(x_{0})\cap\Omega)}\Biggr].$$
\end{lemma}
\begin{proof}
See {\cite[Corollary 3.10 and Theorem 3.11]{Maly}}.
\end{proof}
In the next lemma we state the well known weak Harnack inequality for $p$-Laplacian case. 
\begin{lemma}\label{whi}
Under the assumption \eqref{partcond}. Let $u\in W^{1,p}(\Omega)$ is a nonnegative super-solution of \eqref{probpart}. Then for any ball $B_{4R}(x_{0})\subset\Omega$ and any $q\in (0,\frac{N(p-1)}{N-p})$ there exists $C= C(p,q,r,k,N,\Vert c \Vert_{L^{k}(B_{4R}(x_{0})})>0$ such that  
$$\left(\int_{B_{2R}(x_{0})}u^{q}dx\right)^{\frac{1}{q}}\leq C\Bigr[ \inf_{B_{R}(x_{0})}u + \Vert h^{-} \Vert^{\frac{1}{p-1}}_{L^{k}(B_{4R}(x_{0})}\Bigr].$$
\end{lemma}
\begin{proof}
See {\cite[Theorem 3.13]{Maly}}.
\end{proof}
 For $p=2$ in \cite{Coster1} it was given a new version of Brezis-Cabré lemma. Here we will generalize this new version to the $p$-Laplacian case. The proof of the following lemma is inspired from {\cite[Lemma 2.4]{Coster1}} and from  {\cite[Lemma 3.8]{Abdellaoui}}.  
\begin{lemma}\label{keylemma} 
Let $u\in W^{1,p}(\Omega)$ be a super-solution of \eqref{probpart}. We assume that $c\in L^{\infty}(\Omega)$ and $h\in L^{1}(\Omega)$ be a nonnegative functions. Then for any ball $B_{2R}(x_{0})\subset\Omega$ there exists $C= C(R,N,p,\Vert c\Vert_{\infty})>0$ such that
$$\inf_{\Omega} \left(\frac{u(x)}{d(x,\partial\Omega)}\right)^{p-1}\geq C \int_{B_{R}(x_{0})} h(y)dy$$
\end{lemma}
To prove the above lemma we shall use the following result.  
\begin{lemma}\label{test inequality}
Let $u\in W^{1,p}(\Omega)$ be a nonnegative weak super-solution on $\overline{B_{4R}(x_{0})}\subset\Omega$  of \eqref{probpartbis}.
We assume that $c\in L^{\infty}(\Omega)$ is a nonnegative function. Let $\eta\in C_{0}^{\infty}(B_{2R}(x_{0}))$ be a nonnegative function equal to $1$ on  $B_{R}(x_{0})$. Then there exists $C=C(N,p,R, \Vert c\Vert_{\infty})$ such that 
$$\int_{B_{2R}}\vert\nabla\eta\vert\vert \nabla u\vert^{p-1}\eta^{p-1}dx\leq C R^{N-p}(\inf_{B_{R}}u)^{p-1}.$$
\end{lemma}
\begin{proof}
 See {\cite[Lemma 4.6]{Maly}}
\end{proof}
Next, let us start to  proof the Lemma \ref{keylemma}  
\begin{proof}[Proof of Lemma \ref{keylemma}]
We assumed $h$ is as nonnegative function, then by using weak comparison principle, Lemma \ref{weak comparaison lemma}, we get $u\geq 0$, which imply that $$\inf_{\Omega} \left(\frac{u(x)}{d(x,\partial\Omega)}\right)^{p-1}\geq 0.$$
Let $B_{2R}(x_{0})\subset\Omega$ be fixed, without loss of generality we can assume that
$$\int_{B_{R}(x_{0})}h(x)dx>0.$$
Let $\eta\in C_{0}^{\infty}(B_{2R}(x_{0}))$ be a positive function such that $\eta=1$  on  $B_{R}(x_{0})$. By using $\eta^{p}$ as test function of the problem \eqref{probpart}, we obtain
$$p\int_{\Omega}\eta^{p-1} \vert \nabla u \vert^{p-2}\nabla u \nabla\eta dx + \int_{\Omega}c(x) u ^{p-1}\eta^{p} dx = \int_{\Omega}h(x)\eta^{p},$$
thus
\begin{equation}\label{eq1}
\int_{B_{R}(x_{0})}h(x) dx\leq p\int_{B_{2R}(x_{0})}\eta^{p-1} \vert \nabla u \vert^{p-2}\nabla u \nabla\eta dx + C_{1}\int_{B_{2R}(x_{0})} u ^{p-1} dx.    
\end{equation}
Since the solution of the problem \eqref{probpart} is a nonnegative super-solution of the following problem 
$$
 \begin{cases}                                                              
    -\Delta_{p}u+ c(x)u^{p-1}= 0 & \ \ \mbox{ in }\Omega,\\                                                  
    u=0 & \ \ \mbox{ on }   \partial\Omega.                                                                  
    \end{cases}
$$
Therefore, by using lemma \ref{test inequality} and lemma \ref{whi} we deduce from \eqref{eq1} that  
\begin{equation}\label{h inf}
\int_{B_{R}(x_{0})}h(x) dx\leq  C_{2}(\inf_{B_{R}(x_{0})}u)^{p-1}.    
\end{equation}
Since $B_{2R}(x_{0})\subset\Omega$, then for all $x\in\overline{B_{R}(x_{0})}$ there exist $C_{3}, C_{4}$ two positive constants such that  $C{3}\leq d(x,\partial\Omega)\leq C_{4}$.     
Hence, 
\begin{equation}\label{eq2}
\left( \frac{u(x)}{d(x,\partial\Omega} \right)^{p-1} \geq C_{5}\int_{B_{R}(x_{0})}h(x) dx,  \; \text{ for all } \; x\in\overline{B_{R}(x_{0})}. 
\end{equation}
Next, let $z$ be the unique solution to the problem
\begin{equation}\label{eq3}
\begin{cases}                                                              
    -\Delta_{p}z+ \Vert c \Vert_{\infty} z^{p-1}= 0 & \ \ \mbox{ in }\Omega\textbackslash\overline{B_{R}(x_{0})} ,\\                                                  
    z=0 & \ \ \mbox{ on }   \partial\Omega, \\
    z= 1 & \ \ \mbox{ on }   \partial B_{R}(x_{0}).  
    \end{cases} 
\end{equation}
Therefore from the strong maximum principle \cite{Vazquez} $z$ is positive on $\Omega\textbackslash\overline{B_{R}(x_{0})}$ and due to Hopf's lemma \cite{Vazquez} we deduce that there exists a positive constant $C_{6}$ such that 
\begin{equation}\label{eq4}
z(x)\geq C_{6} d(x,\Omega), \; \;  \forall x\in \Omega\textbackslash\overline{B_{R}(x_{0})}.    
\end{equation}
Now, we define 
$$v(x)= \frac{C^{1/(p-1)}_{2} u(x)}{\left(\int_{B_{R}(x_{0})}h(y)dy \right)^{1/(p-1)}}.$$
From \eqref{h inf} we can easily see that $v(x)\geq 1$ for all $x\in\overline{B_{R}(x_{0})}$ and $-\Delta_{p}v+ \Vert c \Vert_{\infty} v^{p-1}\geq 0$, which imply that $v$ is a super-solution of \eqref{eq3}. Thus by weak comparison principle, Lemma \ref{weak comparaison lemma}, we get that  $v(x)\geq z(x)$ for all $x\in\Omega\textbackslash\overline{B_{R}(x_{0})}$. Thus, from \eqref{eq4} we obtain 
\begin{equation}\label{eq5} 
\left( \frac{u(x)}{d(x,\partial\Omega} \right)^{p-1} \geq C_{7}\int_{B_{R}(x_{0})}h(x) dx,  \; \text{ for all } \; x\in\Omega\textbackslash\overline{B_{R}(x_{0})}. 
\end{equation}
By combining \eqref{eq2} and \eqref{eq5} we deduce the result.
\end{proof}
Let us state the boundary weak Harnack inequality which is established recently in \cite{Coster1}. However this inequality was established in the case of the problem \eqref{probpartbis}  and under stronger condition then \eqref{partcond}.  
\begin{lemma}\label{bwhi}
Let $u\in W^{1,p}(\Omega)$ be a nonnegative super-solution of \eqref{probpartbis}. We assume $c\in L^{\infty}(\Omega)$, is a nonnegative function.  Let $x_{0}\in\partial\Omega$. Then there exist $\overline{R}>0$, $\epsilon= \epsilon(p, \overline{R}, \Vert c\Vert_{L^{\infty}(B_{2R}(x_{0})\cap\Omega)}, \Omega)>0$ and  $C= C(p,\overline{R},\epsilon,\Vert c\Vert_{L^{\infty}(B_{2R}(x_{0})\cap\Omega)},\Omega)>0$  such that for all $R\in(0,\overline{R}],$
$$\left(\int_{B_{R}(x_{0})\cap\Omega}\left(\frac{u(x)}{d(x,\partial\Omega)}\right)^{\epsilon}dx\right)^{\frac{1}{\epsilon}}\leq C\inf_{B_{R}(x_{0})\cap\Omega} \frac{u(x)}{d(x,\partial\Omega)} .$$
\end{lemma}
\begin{proof}
 See {\cite[Theorem 3.1]{Coster1}}.   
\end{proof}
Since we will use the variational argument to prove  the multiplicity of solutions of the problem $(Q)$. let us recall firstly the standard definition of Palais-Smale sequence at  level $c$, Palais-Smale condition at  level $c$ of a functional $I\in C^{1}(E, \mathbb{R})$ where $E$ is a Banach space and some results.
\begin{definition}\label{ Palais-Smale sequence}
        Let $E$ be a Banach space with dual space $E^{\ast}$ and  $\{u_{n}\}$ is a sequence of $E$. We say that $\{u_{n}\}$ is a Palais-Smale sequence at the level $c$ if
        $$ I(u_{n}) \to c,\; \; \mbox{ and }\; \; \|I^{\prime}(u_{n})\|_{E^{*}}\to 0.$$
\end{definition}
\begin{definition}\label{ Palais-Smale condition}
        Let $\{u_{n}\}$ be a Palais-Smale sequence at the level $c$ of $E$. We say that $\{u_{n}\}$ satisfies the Palais-Smale condition at the level $c$  if $\{u_{n}\}$ possesses a convergent subsequence.
\end{definition}
Finally, let us recall some results from \cite{Jeanjean1} which will be very helpful for us. 
\begin{theorem}\label{theorem jeanjean}
Let $E$ be a Banach space with dual space $E^{\ast}$ endowed with the norm $\Vert . \Vert_{E}$. We consider a family  $\{I_{\lambda}\}\subset C^{1}(E, \mathbb{R})$ with the form 
$$I_{\lambda}= A(u)-\lambda B(u), \; \; \forall \lambda\in [1/2,1],$$
where $B(u)\geq 0$ for every $u\in E$ and such that either $A(u)\to +\infty$ or $B(u)\to +\infty$ as $\Vert u\Vert_{E}\to +\infty$.\\
We assume that there exists $v_{0}\in E$ such that with 
$$\Gamma =\{\gamma\in C([0,1], E): \gamma(0)=0,\; \gamma(1)=v_{0}\},$$
we have for all $\lambda\in [1/2,1]$
$$c_{\lambda}= \inf_{\gamma\in \Gamma} \max_{t\in [0,1]} I_{\lambda}(\gamma(t))\geq \max\{ I_{\lambda}(0),I_{\lambda}(v_{0}) \}.$$
Then, for almost every $\lambda\in [1/2,1]$, there exists a bounded Palais-Smale sequence ${v_{n}}\subset E$ at level $c_{\lambda}$ such that 
\begin{enumerate}
    \item ${v_{n}}$ is bounded,
    \item $I_{\lambda}(v_{n})\to c_{\lambda}$,
    \item $I^{'}_{\lambda}(v_{n})\to 0$ in $E^{\ast}$.
\end{enumerate}
\end{theorem} 
\begin{corollary}\label{corollary montain pass}
Let $E, I_{\lambda}$ be as in Theorem \ref{theorem jeanjean} and we assume that both $B$ and $B^{'}$ take bounded sets to bounded sets. Suppose in addition that for all $\lambda\in [1/2,1]$, all bounded Palais-Smale sequences admit a convergence subsequence. Then there exists a subsequence $\{(\lambda_{n}, u_{n})\}\subset [1/2,1]\times E$ with $\lambda_{n}\to 1$, $I_{\lambda_{n}}(u_{n})= c_{\lambda_{n}}$, $I^{'}_{\lambda_{n}}(u_{n})= 0$. Moreover, if $\{u_{n}\}$ is bounded then 
$$I_{\lambda_{n}}(u_{n})\to \lim_{n\to \infty} c_{\lambda_{n}}= c_{1}  \hspace*{0.5cm} I^{'}_{\lambda_{n}}(u_{n})\to 0 \; \; \text{ in } E^{*}.$$
\end{corollary}        
\section{Proof of the main results}
In this section we will give the proofs of our main results. 

\subsection{Proof of a priori bounds}   
\begin{proof}[Proof of Theorem \ref{a priori estimate}]
As we have seen in Lemma \ref{regul}, every weak solution of $(Q_{\lambda})$ belong to $L^{\infty}(\Omega)$. 
Now, let us firstly show that there exists a positive constant $C$ such that
 \begin{equation}
\Vert u \Vert_{L^{\infty}(\Omega_{+})}<C.     
 \end{equation}
 For that we arguing by contradiction. We assume that there exist a sequences $\{\lambda_{n}\}\subset[1/2,1]$ and  $\{u_{n}\}$ nonnegative solutions of $(Q_{\lambda})$ for $\lambda=\lambda_{n}$ such that $\Vert u_{n} \Vert_{L^{\infty}(\Omega_{+})}\to +\infty$  as $n\to +\infty$. Let $\{x_{n}\}\subset\overline{\Omega}_{+}$ be such that
 $u_{n}(x_{n})= \Vert u_{n} \Vert_{L^{\infty}(\Omega_{+})}$. Passing to a sub-sequences we may assume that $\lambda_{n}\to \lambda\in [1/2, 1]$ and $x_{n}\to \Bar{x}\in \overline{\Omega}_{+}$. Therefore, from now we will distinguish the following two cases:  
 \begin{itemize} 
     \item \textbf{ \underline{Case 1}} : $\Bar{x}\in\overline{\Omega}_{+}\cap\Omega$.  
     \item \textbf{ \underline{Case 2}} : $\Bar{x}\in\overline{\Omega}_{+}\cap\partial\Omega$.
 \end{itemize}
 Through this proof we denote $B_{kR}=B_{kR}(\Bar{x})$ where $k\in \mathbb{N}$.
 
 \textbf{ \underline{Case 1}} :  From the assumption \eqref{hypothse} there exists a $R>0$ such that $a^{-}(x)=0$ in $B_{4R}\subset\Omega$ and $a^{+}(x)\gneqq 0$ in $B_{R}$.
 
 If $u_{n}<1$ on $B_{4R}$, there is nothing to prove because we get directly a contradiction. Let us assume that $u_{n}\geq 1$ on $B_{4R}$.\\
 We claim that there exist $R>0$ and $C>0$ such that 
 \begin{equation}\label{eq0}
\inf_{B_{R}} u_{n}<C.     
 \end{equation}
 Otherwise, it is easy to see that 
 \begin{equation} \label{eq6}
-\Delta_{p}u_{n}\geq \lambda a^{+}(x)f(u_{n})-b^{-}(x)g(u_{n})  \; \; \mbox{ in } B_{4R}.     
 \end{equation}
However, from \eqref{limg2} there exists a positive constant $C_{1}$ such that
\begin{equation}\label{eq10}
g(u_{n})\leq C_{1} u^{p-1}_{n} \; \; \text{ in } B_{4R},    
\end{equation}
 it follows from \eqref{eq6} that
 \begin{equation}\label{eq7}
-\Delta_{p}u_{n}+C_{1}b^{-}(x)u^{p-1}_{n} \geq \lambda a^{+}(x)f(u_{n})  \; \; \mbox{ in } B_{4R}.     
 \end{equation}
 Hence, from Lemma \ref{keylemma}, there exists a positive constant $C_{2}$ such that 
$$
\left(\inf_{B_{R}} \frac{u_{n}(x)}{d(x,\partial B_{4R})}\right)^{p-1}\geq \frac{C_{2}}{2} \int_{B_{R}}a^{+}(y)f(u_{n})dy,$$
 which imply
 $$ 1\geq C_{3} \int_{B_{R}}a^{+}(y)\frac{f(u_{n})}{\left( \inf_{B_{R}} u_{n}(x)\right)^{p-1}} dy,$$
then
\begin{equation}\label{eq8}
1\geq C_{3} \int_{B_{R}}a^{+}(y)\frac{f(u_{n})}{u^{p-1}_{n}(y)} dy.   
\end{equation}
Now, if $\inf_{B_{R}} u_{n}(x)\to +\infty \; \; \text{ as } n\to \infty.$ Hence from \eqref{limf2} we obtain $$\int_{B_{R}}a^{+}(y)\frac{f(u_{n})}{ u^{p-1}_{n}(x)} dy\to \infty  \; \; \text{ as } n\to \infty,$$
which contradict with \eqref{eq8}. Therefore, we get the claim. \\
In one hand, we obtain easily from \eqref{eq7} that  
\begin{equation}\label{eq9}
-\Delta_{p}u_{n}+C_{1}b^{-}(x)u^{p-1}_{n} \geq 0  \; \; \mbox{ in } B_{4R},     
 \end{equation}
 then from Lemma  \ref{whi} that, for any $q\in (0,\frac{N(p-1)}{N-p})$ there exists a positive constant $C_{4}$ such that 
 \begin{equation}\label{eqinf}
\left(\int_{B_{2R}}u_{n}^{q}dx\right)^{\frac{1}{q}}\leq C_{4} \inf_{B_{r}}u_{n}.     
 \end{equation}
On the other hand, we have 
\begin{equation}\label{eq11}
-\Delta_{p}u_{n}\leq  a^{+}(x)f(u_{n})+b^{+}(x)g(u_{n}) \; \; \mbox{ in } B_{4R}.      
\end{equation}
Since $f$ is regulary varying of index $p-1$, then,  from Proposition \ref{prop RV}, we have for all $\eta>0$ there exists a positive constant $C_{5}$, which depends on $\eta$, such that  $f(u_{n})\leq C_{5} u^{p-1+\eta}_{n}$. Furthermore,  by \eqref{eq10} and \eqref{eq11}  we get that  
\begin{equation}\label{eq12}
-\Delta_{p}u_{n}\leq  d_{n}(x)u^{p-1}_{n} \; \; \mbox{ in } B_{4R},      
\end{equation}
where $d_{n}(x)= C_{5} a^{+}(x)u^{\eta}_{n} + C_{1}b^{+}(x)$. Next, let show that $d_{n}$ is bounded  in $L^{\beta}(B_{2R})$, where $\beta>\frac{N}{p}$. For that we choose $\eta\in (0, \frac{q}{k}(\frac{pk}{N}-1))$  and we set $\beta=\frac{kq}{q+\eta k}$. Hence by using H\"{o}lder inequality we get 
\begin{align*}
\int_{B_{2R}} (d_{n}(x))^{\beta} dx &\leq C_{6}\left( \int_{B_{2R}} (a^{+}(x)u^{\eta}_{n})^{\beta}dx + \int_{B_{2R}} (b^{+}(x))^{\beta}dx \right)\\
&\leq C_{6}\left(\int_{B_{2R}} (a^{+}(x))^{k}dx\right)^{\beta/k} \left( \int_{B_{2R}} u^{q}_{n}dx\right)^{\beta\eta/q} + C_{7}\left( \int_{B_{2R}} (b^{+}(x))^{k}dx \right)^{\beta/k}.
\end{align*}
By using \eqref{eq0} and \eqref{eqinf} we obtain that 
$$  
\int_{B_{2R}} (d_{n}(x))^{\beta} dx\leq C_{8}. 
$$
Therefore, from Lemma \ref{lmp}, for any $\gamma\in (0,p)$ there exists $C_{9}$ a positive constant such that 
\begin{equation}\label{eqsup}
\sup_{B_{R}} u_{n} \leq C_{9} \left(\int_{B_{2R}}u_{n}^{\gamma}dx\right)^{\frac{1}{\gamma}}.
\end{equation}
By choosing $\gamma=q=p-1$ and by combining  \eqref{eq0}, \eqref{eqinf} and \eqref{eqsup} we deduce that 
$$\sup_{B_{R}} u_{n} \leq C_{10},$$
which imply a contradiction. \\

\textbf{ \underline{Case 2}} :  From the assumption \eqref{hypothse} there exists a $R\in (0, R^{'}/2]$, where $R^{'}$ is given from Lemma \ref{bwhi}, and $\omega\subset\Omega$ with $\partial\omega$ of class $C^{1,1}$ such that $B_{2R}\cap\Omega\subset\omega$,  $a^{-}(x)=0$ and  $a^{+}(x)\gneqq 0$ in $\omega$.

As in the case 1, if $u_{n}<1$ on $\omega$, there is nothing to prove because we get directly a contradiction.  We assume that $u_{n}\geq 1$ on $\omega$.\\ 
Next, let us prove that there exists a positive constant $C$ such that 
\begin{equation}\label{eqinfbound}
\inf_{B_{2R}\cap\omega} \frac{u_{n}(x)}{d(x,\partial\omega)}<C.    
\end{equation}
Let $\Bar{R}$ and $z\in \Omega$ be such that $B_{4\Bar{R}}(z)\subset B_{2R}\cap\omega$ and $a^{+}(x)\gneqq 0$ in $B_{\Bar{R}}$, by using the same arguing in the claim of the case 1, there exists $C_{1}$ a positive constant such that 
\begin{equation}\label{eqinf2}
\inf_{B_{\Bar{R}}} u_{n}(x)<C_{1}.    
\end{equation}
Moreover, since $B_{4\Bar{R}}(z)\subset B_{2R}\cap\omega$, we have 
$$
\inf_{B_{2R}\cap\omega} \frac{u_{n}(x)}{d(x,\partial\omega)}\leq\inf_{B_{\Bar{R}}(z)} \frac{u_{n}(x)}{d(x,\partial\omega)}\leq \frac{1}{\Bar{R}} \inf_{B_{\Bar{R}}(z)} u_{n}(x),
$$
then from \eqref{eqinf2} we deduce directly \eqref{eqinfbound}.  On the other hand, from \eqref{eq7} we have that   
$$
-\Delta_{p}u_{n}+C_{2}b^{-}(x)u^{p-1}_{n} \geq 0  \; \; \mbox{ in } \omega,     
$$
thus from Lemma \ref{bwhi} and the inequality \eqref{eqinfbound} there exist $\epsilon>0$ and $C_{3}>0$ such that
\begin{equation}
\left(\int_{B_{2R}\cap\omega}\left(\frac{u(x)}{d(x,\partial\omega)}\right)^{\epsilon}dx\right)^{\frac{1}{\epsilon}}\leq C_{3},    
\end{equation} 
since  $d(x,\partial\omega)<\text{diam}(\Omega)$, it follows that  
\begin{equation}\label{intlast}
\left(\int_{B_{2R}\cap\omega}\left(u(x)\right)^{\epsilon}dx\right)^{\frac{1}{\epsilon}}\leq C_{3} \text{ diam}(\Omega).    
\end{equation}  
However, we can have as in case 1 that 
$$
-\Delta_{p}u_{n}\leq  d_{n}(x)u^{p-1}_{n}  \; \; \mbox{ in } \omega,      
$$
where $d_{n}(x)=  C_{4} a^{+}(x)u^{\delta}_{n} + C_{5}b^{+}(x)$, for any $\delta>0$. In addition, for this case we choose $\delta\in (0, \frac{\epsilon}{k}(\frac{pk}{N}-1))$  and we set $s_{1}=\frac{k\epsilon}{\epsilon+\delta k}>\frac{N}{p}$. Then we can prove as above that $d_{n}$ is bounded in  $L^{s_{1}}(B_{2R}\cap\omega)$. Therefore from Lemma \ref{blmp}, for any $\beta\in (0,p]$ there exists a positive constant  $C_{6}$, such that
\begin{equation}\label{suplast}
\sup_{B_{R}\cap\omega} u_{n} \leq C_{6} \left(\int_{B_{2R}\cap\omega}u_{n}^{\beta} dx\right)^{\frac{1}{\beta}}.    
\end{equation}
Hence by choosing $\beta=\epsilon$. We combine \eqref{intlast} and \eqref{suplast} then we get
$$
\sup_{B_{R}\cap\omega} u_{n} \leq C_{7},  
$$
which imply the contradiction that we are looking for in this case.

Next, let us show that there is a positive $C>0$ such that 
\begin{equation}\label{priori_cont}
\Vert u \Vert_{L^{\infty}(\Omega\setminus\overline{\Omega}_{+})}<C.     
 \end{equation}
We set $K=\Omega\setminus\overline{\Omega}_{+}$. We have that
$$
-\Delta_{p}u=  -a^{-}(x)f(u)+b(x)g(u)  \; \; \text{ in } K.      
$$
We define  $z= u-\sup_{\partial K}u$, then we get that 
$$ 
-\Delta_{p}z\leq  b^{+}(x)g(z+ \sup_{\partial K}u)  \; \; \text{ in } K.      
$$
From the assumption \eqref{limg2}, we can easily have  that
$$ 
-\Delta_{p}z\leq C_{1} b^{+}(x)(1 + z^{p-1})  \; \; \mbox{ in } K,      
$$
where $C_{1}$ is a positive constant. Therefore $z$ is a sub-solution of the following problem 
\begin{equation}\label{probpart2} 
 \begin{cases}                                                             
    -\Delta_{p}v= C_{1} b^{+}(x)(1 + v^{p-1})  & \ \ \mbox{ in } K,\\                                                  
    v=0 & \ \ \mbox{ on }   \partial K,                                         
    \end{cases}
\end{equation} 
where $v\in W_{0}^{1,p}(K)\cap L^{\infty}(K)$. From {\cite[Theorem  IX-22]{Ladyzhenskaya}} we have $v\in C^{0,\alpha}(\Bar{K})$.  
Thus, from the weak comparison principle in  {\cite[Theorem 3.1]{Coster22}}  we get that 
\begin{equation}\label{eqpartial} 
    u\leq v+\sup_{\partial K}u  \; \text{ in } \; K.
\end{equation}
However, from  {\cite[Theorem 1.2]{Guedda}} there exists a positive constant $C_{2}$ such that $$\Vert v\Vert_{L^{p}(K)}\leq C_{2}.$$
Then, by using Lemma \ref{lmp}  we deduce directly that there exists a positive constant $C_{3}$ such that   
$$\sup_{ K} v \leq C_{3}.$$
Hence, from \eqref{eqpartial}, we obtain that
$$
u\leq C_{3}+\sup_{\Omega_{+}}u  \; \text{ in } \; K,
$$
which imply \eqref{priori_cont}, hence the end of the proof.  
\end{proof}
\subsection{Proof of the multiplicity of solutions}
One of the most fruitful ways to deal with problem $(Q_{\lambda})$, for any $\lambda\in [1/2,1]$, is the variational method, which takes into account that the weak
solutions of $(Q_{\lambda})$ are critical points in  $W_{0}^{1,p}(\Omega)$ of the following functional 
\begin{equation}\label{variational formula}
    I_{\lambda}(u)= \frac{1}{p}\int_{\Omega}\vert\nabla u \vert^{p}-\int_{\Omega} (\lambda a^{+}(x)-a^{-}(x))F(u)- \int_{\Omega} b(x) G(u),   
   \end{equation}
such that $F(s)=\int_{0}^{s}\tilde{f}(s)ds$ and $G(s)=\int_{0}^{s}\tilde{g}(s)ds$, where  $$\tilde{f}(s)= \begin{cases}                                                             
     f(s)& \ \ \mbox{ if } s\geq 0,\\                                                  
    0 & \ \ \mbox{ if }   s< 0,                                                   
    \end{cases}       \hspace{1.5cm} \tilde{g}(s)= \begin{cases}                                                             
     g(s)& \ \ \mbox{ if } s\geq 0,\\                                                  
    g(0) & \ \ \mbox{ if }   s< 0.                                                   
    \end{cases}$$
 
Before we start to proof our main results. It is worth to mention that we will use the same method used in \cite{Chile} where the proof depend on an abstract result due to \cite{Jeanjean1}. This method is based on embedding the functional $I$ into a one-parameter family of functionals $I_{\lambda}$, with $\lambda\in [1/2,1]$, such that $I_{1}=I$. On the other hand, the a priori estimates obtained in the last section will help us to get easily the boundness of the Palais-Smale sequence for $I_{\lambda}$, for all $\lambda\in [1/2,1]$. 

 %We define  the level at $c_{\lambda}$ as follows 
%$$c_{\lambda}= \inf_{\gamma\in \Gamma} \max_{t\in [0,1]} %I_{\lambda}(\gamma(t)),$$
%where
%$\Gamma =\{\gamma\in C([0,1], W_{0}^{1,p}(\Omega)): \gamma(0)=0,\; \gamma(1)=v_{0}\}$ is the set of continuous paths joining $0$ and $v_{0}$, with $v_{0} \in W_{0}^{1,p}(\Omega)$. 

The proof of our main results is as follows. In the first step, we show
the existence of the first critical point for the $I_{1}$ by using the Corollary \ref{corollary montain pass}. Precisely, we show that the functional $I_{1}$
satisfies the Palais-Smale condition at the level $c_{1}>0$. In the
second step, we show the existence of the second critical point
of $I_{1}$ on $B_{\rho}(0)$ (which is a local minimum) by using the lower
semicontinuity argument. Finally, we will prove that these critical points are not the same. 

In the following lemma we prove the strongly convergence of Palais-Smale sequence for $I_{\lambda}$.   
\begin{lemma}\label{convergence Palais-Smale}
Under the assumptions \eqref{hypothse} \eqref{RVq}, \eqref{limg2}, then, for all $\lambda\in [1/2,1]$, any bounded Palais-Smale sequence for $I_{\lambda}$ defined by \eqref{variational formula} admits a convergent subsequence. 
\end{lemma}
\begin{proof}
 Let $\{u_{n}\}\subset W_{0}^{1,p}(\Omega)$ be a bounded Palais-Smale sequence  for $I_{\lambda}$ at level $c>0$.  Then there exists a subsequence denoted again by $\{u_{n}\}$ such that $u_{n}$ converges weakly to $u$ in $W_{0}^{1,p}(\Omega)$, strongly in $L^{r}(\Omega)$ with $r\in [1,p^{*})$ and almost everywhere in $\Omega$. We have that 
 \begin{align*}
 \langle I^{\prime}(u_{n}), u_{n}-u \rangle = \int_{\Omega}\vert\nabla u_{n} \vert^{p-2}\nabla u_{n} \nabla (u_{n}-u)dx -\int_{\Omega} a_{\lambda}(x)\tilde{f}(u_{n})(u_{n}-u)dx - \int_{\Omega} b(x)\tilde{g}(u_{n})(u_{n}-u)dx.   
 \end{align*}
 Let $\delta \in (0, (\frac{p}{N}-\frac{1}{k})p^{*})$, $s\in (1,\frac{p^{*}}{p-1+\delta})$ and $r\in (1,p^{*})$ such that $\frac{1}{k}+\frac{1}{s}+\frac{1}{r}=1$. From Proposition \ref{prop RV}, H\"{o}lder's inequality and Sobolev's embedding, we have 
 \begin{align*}
 \vert \int_{\Omega} a_{\lambda}(x)\tilde{f}(u_{n})(u_{n}-u)dx \vert &\leq  \int_{\Omega} \vert  a_{\lambda}(x)\tilde{f}(u_{n})(u_{n}-u)\vert dx  \\
 & \leq \Vert a\Vert_{k} \Vert \tilde{f}(u_{n}) \Vert_{s}\Vert u_{n}-u \Vert_{r}\\
 &\leq C_{1} \Vert a\Vert_{k} \left(1+ \Vert u_{n}\Vert^{p-1+\delta}_{(p-1+\delta)s}\right) \Vert u_{n}-u \Vert_{r}\\
 &\leq C_{2} \Vert a\Vert_{k} \left(1+ \Vert u_{n}\Vert^{p-1+\delta}\right) \Vert u_{n}-u \Vert_{r}\\
 \end{align*}
 Hence, 
 $$ \int_{\Omega} a_{\lambda}(x)\tilde{f}(u_{n})(u_{n}-u)dx\to 0 \; \;  \text{ as } n\to \infty.$$
 Form the assumption \eqref{limg2}, there $C_{3}$ a positive constant such that 
 $$\vert\tilde{g}(u_{n})\vert\leq C_{3}(1+\vert u \vert^{p-1}_{n}).$$
 By using the same argument we show easily that 
 $$ \int_{\Omega} b(x)\tilde{g}(u_{n})(u_{n}-u)dx\to 0 \; \;  \text{ as } n\to \infty.$$
 Since $\langle I^{\prime}(u_{n}), u_{n}-u \rangle\to 0$ as $n\to \infty$. Thus, we deduce that 
 $$ \int_{\Omega}\vert\nabla u_{n} \vert^{p-2}\nabla u_{n} \nabla (u_{n}-u)dx \to 0  \text{ as } n\to \infty.$$ 
 Therefore, by using {\cite[Theorem 10]{Dinca}}, we conclude that $u_{n}$ converges strongly to $u$ in $W_{0}^{1,p}(\Omega)$.
\end{proof}  
In the following lemma we show the nonnegativity of weak solutions of  $(Q_{\lambda})$.
\begin{lemma}\label{positivity}
Under the assumptions \eqref{hypothse}, \eqref{RVq}, \eqref{limf1} (or \eqref{limf3}), and \eqref{limg2}.
\begin{enumerate}
    \item If $g(0)=0$. Then, every  weak solution of $(Q_{\lambda})$ is nonnegative.
    \item If $g(0)>0$ and $b\gneqq 0$. Then, every  weak solution of $(Q_{\lambda})$ is positive.
\end{enumerate} 
\end{lemma}
\begin{proof} 
\begin{enumerate}
    \item By multiplying the first equation of the problem $(Q_{\lambda})$ by the test function $-u^{-}$ and from the definition of $\tilde{f}$ and $\tilde{g}$, we get that 
\begin{equation}\label{non-negative}
   \int_{\Omega}\vert \nabla u^{-} \vert^{p}dx = - g(0) \int_{\Omega}b(x)u^{-}dx. 
\end{equation}
If $g(0)=0$. Thus,  $\int_{\Omega}\vert \nabla u^{-} \vert^{p}dx=0$ which imply that $u\geq 0$. 
\item Since $g(0)>0$ and $b\gneqq 0$, then, from \eqref{non-negative}  we obtain that  $\int_{\Omega}\vert \nabla u^{-} \vert^{p}dx=0$ which mean $u\geq 0$. Therefore,
$$
-\Delta_{p}u\geq  -a^{-}(x)f(u)  \; \; \text{ in } \Omega.      
$$
By using the Proposition \ref{prop RV} and the assumption \eqref{limf1}(or \eqref{limf3}), we have for all $\delta>0$ there exist a positive constant $C_{1}$ such that   
 \begin{align*}
-\Delta_{p}u\geq d(x)u^{p-1} \; \; \mbox{ in }\Omega, 
 \end{align*}
where $d(x)= -C_{1}a^{-}(x)(1 + u^{\delta})$. Thus, by using Theorem \ref{a priori estimate} we obtain that $d$ is bounded in $L^{k}(\Omega)$. Hence, from Lemma \ref{whi}, we deduce  either $u=0$ or $u>0$. Since $g(0)>0$ and $b\gneqq 0$ we conclude that $u>0$.  
\end{enumerate}    
\end{proof}
\subsubsection{Proof of the first multiplicity result}
\begin{proof}[Proof of Theorem \ref{maintheo1}]
Firstly, we obtain from \eqref{RVq} and \eqref{limf1} that, there exists $C_{1}$ a positive constant and for all small positive $\epsilon$ we have
\begin{equation}\label{f_growth}
 f(t)\leq \epsilon t^{p-1} +  C_{1} t^{p-1+\eta}  \hspace*{0.4 cm} \forall (t,\eta)\in\mathbb{R}_{+}\times(0,\infty),
\end{equation}
and from \eqref{limg2} that, there exist a positive constant $C_{2}$ such that
\begin{equation}\label{g_growth}
g(t)\leq C_{2}(1+ t^{p-1})  \hspace*{0.2 cm} \forall t\in\mathbb{R}_{+}.
\end{equation}
 Then from \eqref{f_growth} and \eqref{g_growth} we get that $I_{\lambda}\in C^{1}(W_{0}^{1,p}(\Omega),\mathbb{R})$, for any $\lambda\in [1/2,1]$, (see for example \cite{Dinca} page 356).
 On the other hand, to use the Theorem \ref{theorem jeanjean} we shall set 
 $$A(u) =  \frac{1}{p}\int_{\Omega}\vert\nabla u \vert^{p}+\int_{\Omega}a^{-}(x)F(u)- \int_{\Omega} b(x) G(u), $$
and
$$B(u)=\lambda \int_{\Omega} a^{+}(x)F(u).$$ 
We have $F$ is a nonnegative function, thus by using \eqref{g_growth}, H\"{o}lder's inequality and Sobolev's embedding we get
 \begin{align*}
A(u)&\geq \frac{1}{p}\int_{\Omega}\vert\nabla u \vert^{p}- \int_{\Omega} b(x)G(u)\\
&\geq \frac{1}{p}\Vert u \Vert^{p} - C_{2}( \Vert b^{+}\Vert_{L^{k}(\Omega)} \Vert u\Vert_{L^{k^{'}}(\Omega)} + \Vert b^{+}\Vert_{L^{k}(\Omega)} \Vert \vert u \vert^{p}\Vert_{L^{k^{'}}(\Omega)})\\
&\geq \Vert u \Vert^{p} \left( \frac{1}{p}- C_{3}\Vert b^{+}\Vert_{L^{k}(\Omega)}( \Vert u\Vert^{1-p} +  1)\right).
\end{align*}
Hence by choosing
\begin{equation}\label{b_constraint} 
\Vert b^{+}\Vert_{L^{k}(\Omega)} < \frac{1}{C_{3}p},    
\end{equation}
we obtain directly that $A(u)\to+\infty$ as $\Vert u\Vert\to +\infty$. 
Next, it is obvious that for all $u\in W_{0}^{1,p}(\Omega)$ we have $B(u)\geq 0$. In addition, $B$ and $B^{'}$ map bounded sets into bounded sets (see for example \cite{Dinca} page 355).    

Now, let us prove that $I_{\lambda}$ has has a geometrical structure. 
From \eqref{f_growth} and by using H\"{o}lder's inequality we get  that
$$\int_{\Omega}  a^{+}(x)F(u)\leq \epsilon  \Vert a^{+}\Vert_{L^{k}(\Omega)} \Vert \vert u\vert^{p}\Vert_{L^{k^{'}}(\Omega)}+ C_{1} \Vert a^{+}\Vert_{L^{k}(\Omega)} \Vert \vert u\vert^{p+\eta}\Vert_{L^{k^{'}}(\Omega)}.$$
Due to the assumption $k>N/p$, we choose $\eta>0$ small enough such that $(p+\eta)k^{'}<\frac{pN}{N-p}$. Hence, by using Sobolev's embedding we obtain 
\begin{equation}\label{Fequa}
\int_{\Omega}  a^{+}(x)F(u)\leq C_{4} \epsilon  \Vert a^{+}\Vert_{L^{k}(\Omega)} \Vert u\Vert^{p}+ C_{5} \Vert a^{+}\Vert_{L^{k}(\Omega)} \Vert u\Vert^{p+\eta}.  
\end{equation}
Furthermore, as above we get from \eqref{g_growth} that 
\begin{equation}\label{Gequa}
\int_{\Omega} b^{+}(x) G(u) \leq C_{3}\Vert b^{+}\Vert_{L^{k}(\Omega)}( \Vert u\Vert + \Vert u\Vert^{p}).
\end{equation} 
Since $F$ and $G$ are nonnegative thus from \eqref{Fequa} and \eqref{Gequa} we deduce that
\begin{align*} 
I_{\lambda}(u)&\geq \frac{1}{p}\int_{\Omega}\vert\nabla u \vert^{p}-\int_{\Omega}  a^{+}(x)F(u)- \int_{\Omega} b^{+}(x) G(u)\\
&\geq \left[ \frac{1}{p} -   \Vert a^{+}\Vert_{L^{k}(\Omega)}(C_{4}\epsilon + C_{5} \Vert u\Vert^{\eta})-  C_{3}\Vert b^{+}\Vert_{L^{k}(\Omega)} (\Vert u\Vert^{1-p} + 1) \right]\Vert u\Vert^{p}.    
\end{align*}
 Therefore, let $u\in \partial B_{\rho}(0)$ and by choosing $\rho$ small enough  we choose  
$$ \Vert b^{+}\Vert_{L^{k}(\Omega)}<\frac{1}{2p C_{3}(\rho^{1-p} + 1)},$$
we deduce $I_{\lambda}(u)\geq \nu >0$, for all $\lambda\in [1/2,1]$.

Next, we choose $y\in \Omega_{+}$ and $R>0$ such that $B_{R}(y)\subset\Omega_{+}$. Let  $v$ be a positive function that belongs to $C^{\infty}_{0}(B_{R}(y))$ such that $ a^{+} v\gneqq 0$. From the definition of $I_{\lambda}$, we have  
\begin{align*}
I_{\lambda}(tv)&= \frac{t^{p}}{p}\int_{\Omega}\vert\nabla v \vert^{p}-\int_{\Omega} (\lambda a^{+}(x)-a^{-}(x))F(tv)- \int_{\Omega} b(x) G(tv)\\
&=t^{p}\left( \frac{1}{p}\int_{B_{R}(y)}\vert\nabla v \vert^{p}-\lambda\int_{B_{R}(y)} a^{+}(x)\frac{F(tv)}{t^{p}v^{p}}v^{p}- \int_{B_{R}(y)} b(x) \frac{G(tv)}{t^{p}v^{p}}v^{p}\right).
\end{align*}
By using \eqref{limg2} there exists a positive constant $C_{6}$ such that 
$$\int_{B_{R}(y)}\vert b(x) \frac{G(tv)}{t^{p}v^{p}}v^{p}\vert \leq C_{6} \Vert b \Vert_{L^{k}(B_{R}(y))} \; \text{ for $t$ large enough },$$
and from \eqref{limf2} we obtain
$$ \int_{B_{R}(y)} a^{+}(x)\frac{F(tv)}{t^{p}v^{p}}v^{p} \to +\infty \;   \text{ as } \; \; t\to +\infty.$$
Hence by choosing $v_{0}=tv$, defined in Theorem \ref{theorem jeanjean}, with $t$ large enough we obtain that $I_{\lambda}(v_{0})<0$. Therefore, from Theorem \ref{theorem jeanjean} we deduce that $I_{\lambda}$ admit a Palais-Smale sequence $\{u_{n}\}$ at level $c_{\lambda}$. Thus, by using Lemma \ref{convergence Palais-Smale} we have that all bounded Palais-Smale sequences admit a convergent subsequence. Then from Corollary \ref{corollary montain pass} there exist a sequences $\{\lambda_{n}\}\subset [\frac{1}{2},1]$ and $\{u_{n}\}\subset W_{0}^{1,p}(\Omega)$ such that $\lambda_{n}\to 1$, $I_{\lambda_{n}}(u_{n})= c_{\lambda_{n}}$ and $I^{'}_{\lambda_{n}}(u_{n})= 0$, which means that $u_{n}$ is a weak solution of the following problem  

$$(Q_{\lambda_{n}}) \begin{cases}                                                             
    -\Delta_{p}u= a_{\lambda_{n}}(x)\tilde{f}(u)+b(x)\tilde{g}(u) & \ \ \mbox{ in }\Omega,\\                                                  
    u=0 & \ \ \mbox{ on }   \partial\Omega.                                                                 
    \end{cases}$$
From Lemma \ref{regul}, we have that $u_{n}\in L^{\infty}(\Omega)$. Moreover, from Lemma \ref{positivity} we obtain that $u_{n}$ is nonnegative.

 Beside, from Theorem \ref{a priori estimate} there exist a positive constant $C$ such that $\Vert u_{n}\Vert_{L^{\infty}(\Omega)}<C$. Then by multiplying the first equation of problem $(Q_{\lambda_{n}})$ by $u_{n}$ and by using Proposition \ref{prop RV} and \eqref{g_growth} we get that
 \begin{align*}
 \int_{\Omega}\vert\nabla u_{n} \vert^{p}dx &=\int_{\Omega} a_{\lambda_{n}}(x)\tilde{f}(u_{n})u_{n}dx + \int_{\Omega} b(x)\tilde{g}(u_{n})u_{n}dx\\
&\leq C_{7} \left(\int_{\Omega}\vert a(x) \vert (\vert u_{n} \vert +\vert u\vert^{p+\delta}_{n})dx + \int_{\Omega}\vert b(x) \vert \vert (\vert u_{n}\vert +\vert u\vert ^{p}_{n})dx\right)\\
&\leq C_{8} \left(\Vert a\Vert_{k} + \Vert b\Vert_{k}\right).
 \end{align*}
  Therefore, $\{u_{n}\}$ is  bounded in $W_{0}^{1,p}(\Omega)$. Hence, we can find a subsequence,
still denoted $\{u_{n}\}$, such that $u_{n}$ converges weakly to $u$ in $W_{0}^{1,p}(\Omega)$, strongly in $L^{r}(\Omega)$ with $r\in [1,p^{*})$ and almost everywhere to $u$ in $\Omega$. From Lemma \ref{convergence Palais-Smale} we have $u_{n}$ converges strongly to $u$ in $W_{0}^{1,p}(\Omega)$.    
Therefore,  we deduce directly  that $u\in W_{0}^{1,p}(\Omega)$ is a nonnegative  weak solution of the problem $(Q)$ and from Corollary \ref{corollary montain pass} we have $I_{1}(u)=c_{1}>0$ which mean that $u$ is nontrivial.  

Next, let us show the existence of the second weak solution of the problem $(Q)$. Firstly, we have $I(0)=0$ then $\inf_{v\in B_{\rho}(0)}I(v)\leq 0$, where $\rho$ is defined as above. On the other hand, let $\phi\in C_{0}^{\infty}(\Omega)$ be a positive function  such that $a\phi>0$ and $b\phi>0$. From the definition of $I$, for all $t>0$, we have    
\begin{align*}
I(t\phi) &=t^{p}\left( \frac{1}{p}\int_{\Omega}\vert\nabla \phi \vert^{p}-\int_{\Omega} a(x)\frac{F(t\phi)}{t^{p}\phi^{p}}\phi^{p}- \int_{\Omega} b(x) \frac{G(t\phi)}{t^{p}\phi^{p}}\phi^{p}\right).
\end{align*} 
By using \eqref{limf1} and \eqref{limg1},  we get for $t\to 0^{+}$ that
 $$\int_{\Omega} a(x)\frac{F(t\phi)}{t^{p}\phi^{p}}\phi^{r+1}\to 0,$$
 and 
 $$ \int_{\Omega} b(x) \frac{G(t\phi)}{t^{p}\phi^{p}}\phi^{p}\to +\infty.$$
Hence, $I(t\phi)<0$ for $t$ small enough. Then  $\inf_{v\in B_{\rho}(0)}I(v)< 0$. Moreover, it have been shown that that $I(v)\geq \nu >0$ with $\Vert v\Vert= \rho$. Then there exists a sequence $\{v_{n}\}\subset B_{\rho}(0)$ such that $I(v_{n})$ converges to $\inf_{v\in B_{\rho}(0)}I(v)$. Since  $\{v_{n}\}$ is bounded in $W_{0}^{1,p}(\Omega)$ then there exists a subsequence denoted again $\{v_{n}\}$  such that $v_{n}$ converges to $v$ weakly in $W_{0}^{1,p}(\Omega)$ and converges strongly in $L^{r}(\Omega)$ for some $r\in [1,p^{*})$ respectively.  In addition, since $\Vert v \Vert \leq \liminf_{n\to \infty}\Vert v_{n}\Vert$. Hence, we obtain $I(v)\leq \inf_{v\in B_{\rho}(0)}I(v)$. Therefore, $v$ is a nonnegative and nontrivial local minimum of $I$ in  $B_{\rho}(0)$. 

We conclude that $(Q)$ have one weak solution $u$ at level $c_{1}$, which mean that  $I(u)=c_{1}>0$ and a second weak solution $v\in B_{\rho}(0)$ such that $I(v)<0$. Therefore, the problem $(Q)$  has at least two distinct nonnegative and nontrivial weak solutions in $W_{0}^{1,p}(\Omega)\cap L^{\infty}(\Omega)$. Moreover, by using Lemma \ref{positivity} we conclude the result. 
\end{proof}
\subsubsection{Proof of the second multiplicity result}
\begin{proof}[Proof of Theorem \ref{maintheo2}]
 As it was quoted before every weak solution of the problem $(Q_{\lambda})$ is a critical point of the variational formula  \eqref{variational formula}. By the same arguing used in the proof of Theorem \ref{maintheo1} we have $I_{\lambda}\in C^{1}(W_{0}^{1,p}(\Omega), \mathbb{R})$. Let us show that $I_{\lambda}$ has a geometrical structure for avery $\lambda\in [1/2,1]$. 
 
From Proposition \ref{prop RV} and due to \eqref{limf3}, there exist  $C_{1}$ and $C_{2}$ two positive constants such that $f(t)\leq  C_{1}t^{r} +  C_{2} t^{p-1+\eta}, $ for any $\eta>0$. Hence, by choosing $\eta$ small enough, by using H\"{o}lder's inequality and Sobolev's embedding we obtain that
 \begin{align*}
 \int_{\Omega}  a_{\lambda}(x)F(u) & \leq C_{3}(  \Vert a^{+}\Vert_{L^{k}(\Omega)} \Vert u\Vert^{r+1}+ \Vert a^{+}\Vert_{L^{k}(\Omega)} \Vert u\Vert^{p+\eta}). 
 \end{align*}
 Moreover, from \eqref{g_growth} there exists a positive constant $C_{4}$ such that $g(t)\leq C_{4}(1 + t^{p-1})$. As above we have 
$$\int_{\Omega} b(x) G(u)\leq C_{5}(\Vert b^{+}\Vert_{L^{k}(\Omega)} \Vert u\Vert +  \Vert b^{+}\Vert_{L^{k}(\Omega)} \Vert u\Vert^{p}).$$
From the definition of $I_{\lambda}$ we get that
$$
I_{\lambda}(u)\geq \Vert u \Vert^{p} \left( \frac{1}{p}- C_{3}\Vert a^{+}\Vert_{L^{k}(\Omega)}( \Vert u\Vert^{r+1-p} + \Vert u\Vert^{\eta}) - C_{5}\Vert b^{+}\Vert_{L^{k}(\Omega)}( \Vert u\Vert^{1-p} + 1)\right).    
$$
Therefore, let $u\in \partial B_{\rho}(0)$, by choosing $\rho$ small enough, 
$$ \Vert b^{+}\Vert_{L^{k}(\Omega)}< \frac{1}{3p C_{5}(\rho^{1-p} + 1)},$$
and $$ \Vert a^{+}\Vert_{L^{k}(\Omega)}\leq \frac{1}{3pC_{3}(\rho^{r+1-p}+ \rho^{\eta})}.$$  
We deduce that $I_{\lambda}(u)\geq \nu_{2} >0$ on  $\partial B_{\rho}(0)$, for any $\lambda\in [1/2,1]$. By using the same arguing to proof Theorem \ref{maintheo1} we show the existence of the first nonnegative and nontrivial weak solution $u\in W_{0}^{1,p}(\Omega)\cap L^{\infty}(\Omega)$ at level $c_{1}>0$. 

To show that the second weak nonnegative solution  of $(Q)$ on $B_{\rho}(0)$. Let $\phi\in C_{0}^{\infty}(\Omega)$ be a positive function  such that $a_{\lambda}\phi>0$ and $b\phi>0$. From the definition of $I_{\lambda}$ for all $t>0$ we have   
\begin{align*}
I_{\lambda}(t\phi) &=t^{p}\left( \frac{1}{p}\int_{\Omega}\vert\nabla \phi \vert^{p}- t^{r+1-p}\int_{\Omega} a_{\lambda}(x)\frac{F(t\phi)}{t^{r+1}\phi^{r+1}}\phi^{r+1}- \int_{\Omega} b(x) \frac{G(t\phi)}{t^{p}\phi^{p}}\phi^{p}\right).
\end{align*}
 From \eqref{limf3} and \eqref{limg3}, we obtain that 
 $$t^{r+1-p}\int_{\Omega} a_{\lambda}(x)\frac{F(t\phi)}{t^{r+1}\phi^{r+1}}\phi^{r+1}\to +\infty$$
 and 
 $$ \int_{\Omega} b(x) \frac{G(t\phi)}{t^{p}\phi^{p}}\phi^{p}\to +\infty, \hspace*{0.2cm} \text{ or } \hspace*{0.2cm} \int_{\Omega} b(x) \frac{G(t\phi)}{t^{p}\phi^{p}}\phi^{p}\to c <+\infty, $$
 as $t\to 0^{+}$. It follows that $I_{\lambda}(t\phi)<0$ for $t$ small enough. For the rest of the proof is still the same as the proof of Theorem \ref{maintheo1}. 
\end{proof}
\section{Applications}  
 The aim of this section is to give a class of problems that have the same form of the problem $(Q)$, where the conditions \eqref{RVq}-\eqref{limf1} (-\eqref{limf3}) and \eqref{limg2}-\eqref{limg1} (-\eqref{limg3}) are satisfied and for which Theorems \ref{maintheo1} and \ref{maintheo2} hold. In particular, we give two applications for both Theorem \ref{maintheo1} and Theorem \ref{maintheo2}: the first application concerns the case where $g(0)=0$ and the second one, deals with the case where $g(0)>0$ is satisfied.
\subsection{Application 1: Quasilinear equation with a potential}
Let us consider the following problem: 
$$(P_{1}) 
\begin{cases}
-\Delta_{p}u=a(x)(1+u)^{p-1}(\ln(1+u))^{q} +  b(x)u^{m} & \ \ \mbox{ in }\Omega,\\
u=0 & \ \ \mbox{ on }   \partial\Omega,  
\end{cases}$$
where $q>0$ and $m> 0$. This type of problem have been considered in several works in the case where $p=2$, see for instance, \cite{Cirstea0,Cirstea}. To the best of our knowledge there is no multiplicity result in this type of problem in the case where $a$ and $b$ change of sign and $p>1$. 
 
The following result is an application of Theorem \ref{maintheo1}.
\begin{corollary}
 Under the assumption  \eqref{hypothse}. Let $q>p-1$, $ 0< m<p-1$, and $\Vert b^{+} \Vert_{L^{k}(\Omega)}$ is suitably small.  Then, the problem $(P_{1})$  has at least two nonnegative nontrivial weak solutions in $W_{0}^{1,p}(\Omega)\cap L^{\infty}(\Omega)$.  
\end{corollary}
\begin{proof}
We set $f(s) = (1+s)^{p-1}(\ln(1+s))^{q}$ and $g(s)= s^{m}$. We can easily show  that $f$ satisfies the assumptions \eqref{RVq}, \eqref{limf1} and \eqref{limf2} if $q>p-1$. The function $g$ satisfies \eqref{limg2} and \eqref{limg1} if $m\in (0, p-1)$. Hence by using Theorem \ref{maintheo1} we obtain directly the result.   
\end{proof}
The following result is an application of Theorem \ref{maintheo2}.
\begin{corollary}
 Under the assumptions  \eqref{hypothse}. Let $0<q< p-1$, $0 < m\leq p-1$, $\Vert a^{+} \Vert_{L^{k}(\Omega)}$ and $\Vert b^{+} \Vert_{L^{k}(\Omega)}$ be suitably small.  Then, the problem $(P_{1})$  has at least two nonnegative nontrivial weak solutions in $W_{0}^{1,p}(\Omega)\cap L^{\infty}(\Omega)$. 
\end{corollary}
\begin{proof}
As in the above proof, we set $f(s) = (1+s)^{p-1}(\ln(1+s))^{q}$ and $g(s)= s^{m}$. We can easily show  that $f$ satisfies the assumptions \eqref{RVq}, \eqref{limf1} and \eqref{limf3} if $0<q<p-1$. The function $g$ satisfies \eqref{limg2} and \eqref{limg3} if $m\in (0, p-1]$. Hence by using Theorem \ref{maintheo2} we obtain directly the result.
\end{proof}
\subsection{Application 2: Quasilinear equations with a $p$-Gradient term}
We consider the following problem 
$$(P_{2}) 
\begin{cases}
-\Delta_{p}u= a(x)u^{q}+|\nabla u|^{p}+b(x) & \ \ \mbox{ in }\Omega,\\
u=0 & \ \ \mbox{ on }   \partial\Omega,  
\end{cases}$$
where $q>0$. This type of problem is now a very active field of research and it was investigated in various papers, for instance, \cite{Chaouai,Coster22,Jeanjean2, Jeanjean3} and the references therein. It is worth to mention that until now there is no result concerning the multiplicity result in the case where $a$ changes of sign and for the problem $(P_{2})$ with  $p>1$.  

%under the condition $a(x)\leq -\alpha_{0}$ a.e. in $\Omega$\; for some $\alpha_{0}>0$, which is referred as the coercive case, the existence results were given in  \cite{Boccardo1,Boccardo2,Boccardo3} even for more general framework. For the uniqueness it were given in \cite{Barles1, Barles2}. In the weakly coercive case i.e. $a\equiv 0$, there had been many contributions see for example (\cite{Abdellaoui1, Maderna, Poretta}).   Moreover, for the case $p=2$ and in the limit coercive case i.e.  $a(x)\leq 0$  and  may vanish on some parts of $\Omega$, the existence result was left open until the paper \cite{Arcoya1} and the uniqueness was also done in the same paper. For a related uniqueness result in more general framework see also \cite{Arcoya2}. In the non coercive case, i.e.  $a(x)\gneqq 0$ a.e. in $\Omega$, for $p=2$  in \cite{Jeanjean3} it was showed for the first time the existence of two bounded solutions by assuming that $a$ and  $b$ are small enough. In this work $b$ is allowed to change sign. This result was extended in \cite{Chaouai, Coster22} to the $p$-Laplacian case. Finally, in  the case where $a$ is allowed to change sign  with $a^{+}\not\equiv 0$ in \cite{Jeanjean2} it was showed the existence of two bounded positive solutions  when $ b\gneqq 0$ and   $a^{+}$ and $b$ are suitably small. 

We observe that the problem $(P_{2})$ has no relation with the problem $(Q)$ due to  the presence of the $p$-gradient term. However, if we perform the following Kazdan-Kramer change of variable,
\begin{equation}\label{Kramer}
v=e^{\frac{ u}{p-1}}-1.    
\end{equation} 
We obtain  the following problem  
$$(Q_{2}) 
\begin{cases}
-\Delta_{p}v= a(x)f(v)+ b(x)g(v) & \ \ \mbox{ in }\Omega,\\
v=0 & \ \ \mbox{ on }   \partial\Omega,  
\end{cases}$$ 
where
\begin{equation}\label{function f}
f(s)= (p-1)^{q-p+1}(1+ s)^{p-1}(\ln(1+s))^{q},
\end{equation} 
and  
\begin{equation}\label{function g}
g(s)=\frac{(1+ s)^{p-1}}{(p-1)^{p-1}},
\end{equation}	
for all $s\geq 0$.\\

In the following lemma we will prove that, if $u$ is a nonnegative weak solution of $(P_{2})$ then $v$ defined by \eqref{Kramer} is a nonnegative weak solution of $(Q_{2})$. In other words, to show the existence of at least two nonnegative weak solution of $(P_{2})$ it is equivalent to show the existence of at least two nonnegative weak solution of $(Q_{2})$. 
\begin{lemma}\label{change of variable}
Under the assumption \eqref{hypothse}, let $b\gneqq 0$, then 
\begin{enumerate}
\item Every weak solution of $(Q_{2})$ is positive.
\item If $v\in W_{0}^{1,p}(\Omega)\cap L^{\infty}(\Omega)$ is a  positive weak solution of $(Q')$ if only if $$u=(p-1)\ln(1+v)\in W_{0}^{1,p}(\Omega)\cap L^{\infty}(\Omega)$$ is a positive weak solution of $(P_{2})$.
\end{enumerate}
\end{lemma}
\begin{proof}
\begin{enumerate}
    \item See Lemma \ref{positivity}. 
    \item Let $v$ be a positive solution of $(Q_{2})$. From the expression of $f$ and $g$ it is seen that $v$ is a solution of 
    \begin{equation}\label{example equation}
    -\Delta_{p}v= a(x) (p-1)^{q-p+1}(1+ v)^{p-1}(\ln(1+ v))^{q}+ b(x)\frac{(1+ v)^{p-1}}{(p-1)^{p-1}}. 
    \end{equation}
    Let $u=(p-1)\ln(1+v)$, then $\nabla u =(p-1) \frac{\nabla v}{1+v}$, since $v$ is positive, we get easly from \eqref{Kramer} that $u\in W_{0}^{1,p}(\Omega)\cap L^{\infty}(\Omega)$ and positive. Let $\phi\in W_{0}^{1,p}(\Omega)$, we set $\varphi= \frac{(p-1)^{p-1}\phi}{(1+v)^{p-1}}$, then $\varphi\in W_{0}^{1,p}(\Omega)$.
    Hence, from \eqref{example equation}, we have. 
    \begin{equation}
    \int_{\Omega}\vert \nabla v \vert^{p-2}\nabla v \nabla \varphi = (p-1)^{q-p+1}\int_{\Omega}a(x) (1+ v)^{p-1}(\ln(1+ v))^{q}\varphi + \int_{\Omega}b(x)\frac{(1+ v)^{p-1}}{(p-1)^{p-1}}\varphi.    
    \end{equation}
    Since $$\nabla\varphi = \frac{(p-1)^{p-1}\nabla\phi}{(1+v)^{p-1}}- \frac{(p-1)^{p}\phi\nabla v}{(1+v)^{p}},$$
and 
$$\nabla v = \frac{\nabla u}{p-1}e^{u/(p-1)}.$$
Then 
\begin{align}\label{c1}
\int_{\Omega}\vert \nabla v \vert^{p-2}\nabla v \nabla \varphi &=   \nonumber \int_{\Omega}\frac{e^{u}}{(p-1)^{p-1}} \vert \nabla u \vert^{p-2}\nabla u\left(\frac{(p-1)^{p-1}\nabla\phi}{(1+v)^{p-1}}- \frac{(p-1)^{p}\phi\nabla v}{(1+v)^{p}}\right)\\ \nonumber
&=\int_{\Omega}\vert \nabla u \vert^{p-2}\nabla u\left(\nabla\phi- \frac{(p-1)\phi\nabla v}{1+v}\right)\\  
&= \int_{\Omega}\vert \nabla u \vert^{p-2}\nabla u\nabla\phi - \int_{\Omega}\vert\nabla u\vert^{p}\phi. 
\end{align}
On the other hand we have
\begin{equation}\label{c2}
(p-1)^{q-p+1}\int_{\Omega}a(x) (1+ v)^{p-1}(\ln(1+ v))^{q}\varphi=\int_{\Omega}a(x)u^{q}\phi,    
\end{equation}
and 
\begin{equation}\label{c3}
\int_{\Omega}b(x)\frac{(1+ v)^{p-1}}{(p-1)^{p-1}}\varphi =\int_{\Omega}b(x)\phi.     
\end{equation}
By combining \eqref{c1}, \eqref{c2} and \eqref{c3} we deduce that 
$$\int_{\Omega}\vert \nabla u \vert^{p-2}\nabla u\nabla\phi=\int_{\Omega}a(x)u^{q}\phi +\int_{\Omega}\vert\nabla u\vert^{p}\phi+ \int_{\Omega}b(x)\phi,$$
which mean that $u$ is a positive weak solution of $(P_{2})$. By the same arguments we prove the reverse statement. 
\end{enumerate}
\end{proof}

Finally, let us give applications of Theorem \ref{maintheo1} and Theorem \ref{maintheo2} in the case where $g(0)>0$ and $b\gneqq 0$.  

The following result is an application of Theorem \ref{maintheo1}.
\begin{corollary} 
 Under the assumptions  \eqref{hypothse}, let $q>p-1$ and $b\gneqq 0$. If $\Vert b \Vert_{L^{k}(\Omega)}$ is suitably small, then, the problem $(Q_{2})$  has at least two positive weak solutions in $W_{0}^{1,p}(\Omega)\cap L^{\infty}(\Omega)$.  
\end{corollary}

The following result is an application of Theorem \ref{maintheo2}.
\begin{corollary}
 Under the assumptions  \eqref{hypothse}, let $0<q\leq p-1$ and $b\gneqq 0$. If  $\Vert a^{+} \Vert_{L^{k}(\Omega)}$ and $\Vert b \Vert_{L^{k}(\Omega)}$ are suitably small,  then the problem $(Q_{2})$  has at least two positive weak solutions in $W_{0}^{1,p}(\Omega)\cap L^{\infty}(\Omega)$.   
\end{corollary}

\end{document}